\documentclass[12pt]{article}

\usepackage[margin=1in]{geometry}

\usepackage{amsmath}
\usepackage{amsfonts}
\usepackage{amsthm}
\usepackage{mathrsfs}
\usepackage{amssymb}
\usepackage{tikz}
\usepackage{comment}
\usetikzlibrary{calc}
\usepackage{tikz-cd}
\usepackage{graphicx}
\usepackage{verbatim}
\usepackage{ tipa }
\usepackage[nodisplayskipstretch]{setspace}
\usepackage{varwidth}
\usepackage{dsfont}
\usepackage{mathtools}
\usetikzlibrary{decorations.pathreplacing,calligraphy}

\DeclareMathOperator{\Proj}{Proj}
\DeclareMathOperator{\Pic}{Pic}
\DeclareMathOperator{\Spec}{Spec}

\renewcommand{\P}{\mathbb P}

\newcommand{\Oo}{\mathcal O}

\newcommand{\bfa}{\mathbf a}
\newcommand{\bfb}{\mathbf b}

\newcommand{\V}{\mathbf V}

\newcommand{\Q}{\mathbb Q}

\newcommand{\R}{\mathbb R}

\newcommand{\Z}{\mathbb Z}
\newcommand{\C}{\mathbb C}

\newtheorem{theorem}{Theorem}[section]
\newtheorem{lemma}[theorem]{Lemma}
\newtheorem{corollary}[theorem]{Corollary}

\newtheorem{proposition}[theorem]{Proposition}

\theoremstyle{definition}

\newtheorem{example}[theorem]{Example}

\theoremstyle{remark}

\newcommand{\corner}[1]{%
\,
\begin{tikzpicture}[#1]%
\draw[semithick] (1ex,0) -- (0,0) -- (0,1ex);%
\end{tikzpicture}%
\,
}

\renewcommand{\square}{%
\,
\begin{tikzpicture}%
\draw[semithick] (0,0) -- (0,1ex) -- (1ex,1ex) -- (1ex,0) -- cycle;%
\end{tikzpicture}%
\,
}

\title{Verlinde Series for Hirzebruch Surfaces}
\author{Ian Cavey}
\date{\today}

\begin{document}

\maketitle

\begin{abstract}
    We give an explicit formula for Euler characteristics of line bundles on the Hilbert scheme of points on $\P^1\times\P^1$. Combined with structural results of Ellingsrud, Göttsche, and Lehn \cite{EllingsrudGöttscheLehn}, this determines the Euler characteristic of any line bundle on the Hilbert scheme of points on any smooth, projective surface. We also give an enumerative description of the dimensions of spaces of global sections of ample line bundles on Hilbert schemes of points on Hirzebruch surfaces, extending the polytope-line bundle correspondence on the underlying toric surface. 
\end{abstract}

\section{Introduction}

Hilbert schemes of points on surfaces are fundamental examples of moduli spaces in algebraic geometry with connections to a wide variety of topics in math, as well as theoretical physics (for a brief survey see \cite{G}). Verlinde series are central objects of study in the enumerative geometry of these Hilbert schemes. For a smooth, projective surface $X$, let $X^{[n]}$ denote the Hilbert scheme of $n$ points on $X$. This Hilbert scheme is a smooth, projective variety of dimension $2n$, and can be thought of as a compactification of the set of unordered $n$-tuples of distinct points in $X$. Given a line bundle $L$ on $X$, there is an induced line bundle $L_n$ on $X^{[n]}$ pulled back from the symmetric power (see Section \ref{sec:torictoaffine} for the precise definitions). Verlinde series, introduced by Ellingsrud, Göttsche, and Lehn \cite{EllingsrudGöttscheLehn}, are the generating functions for holomorphic Euler characteristics of line bundles,
\[ \V_{X,L,r}(z) = \sum_{n= 0}^\infty z^n\cdot  \chi(X^{[n]},L_n\otimes E^r) \in \Z[[z]], \]
where $E$ is $-1/2$ times the exceptional divisor on $X^{[n]}$ and $r$ is a fixed integer.\\

Verlinde series contain fundamental enumerative information about Hilbert schemes of points. All line bundles on $X^{[n]}$ are of the form $L_n\otimes E^r$ \cite{FogartyII}, so the coefficients of Verlinde series contain the holomorphic Euler characteristics of all line bundles. In particular, sufficiently ample line bundles on $X^{[n]}$ correspond to projective embeddings $X^{[n]}\hookrightarrow \P(H^0(X^{[n]},L_n\otimes E^r))$, and in this case the coefficient on the $z^n$ term of the Verlinde series $\V_{X,L,r}(z)$ is the dimension of the vector space $H^0(X^{[n]},L_n\otimes E^r).$\\

Verlinde series are also known to be related to other important enumerative invariants of Hilbert schemes. Johnson \cite{J} and Marian, Oprea, and Pandharipande \cite{MOPCombofLehn} conjectured that Verlinde series can be transformed into the generating series for the top Segre classes of higher rank tautological vector bundles on $X^{[n]}$ by an explicit change of variables. These higher rank Segre series generalize those in Lehn's well-known conjecture \cite{Lehn}, which was proved by Voison \cite{Voisin} and Marian, Oprea, and Pandharipande \cite{MOPCombofLehn}. Higher rank Segre series are known completely for $K$-trivial surfaces \cite{MOPSegre}, while formulas for arbitrary surfaces are known only for tautological vector bundles of certain ranks \cite{MOPSegre,Y}. The conjectured correspondance between Verlinde series and higher rank Segre series was recently proved by Göttsche and Mellit \cite{GM}, although neither series was determined in general. Consequently, the determination of general formulas for either the Verlinde series or higher rank Segre series also determines the other. \\

An essential feature of Verlinde series established by Ellingsrud, Göttsche, and Lehn \cite{EllingsrudGöttscheLehn} is the factorization into universal power series $A_r,B_r,C_r,D_r\in \Q[[z]]$ depending only on $r$,
\begin{equation}\label{eqn:factorization} \V_{X,L,r}(z) = A_r(z)^{\chi(L)}\cdot B_r(z)^{\chi(\Oo_X)}\cdot C_r(z)^{c_1(L)\cdot K_X-\frac12 K_X^2}\cdot D_r(z)^{K_X^2}, \end{equation}
where $K_X$ is the canonical divisor on $X$. In particular, $\chi(X^{[n]},L_n\otimes E^r)$ does not depend on the pair $(X,L)$ beyond the four enumerative invariants $\chi(L),\chi(\Oo_X),c_1(L)\cdot K_X,$ and $K_X^2$. Computing Verlinde series can therefore be reduced to finding formulas for the series $A_r,B_r,C_r$ and $D_r$ for each integer $r.$ Using Serre duality one can show that these series satisfy the symmetry relations $A_{-r}(z)=A_r(z),$ $B_{-r}(z)=B_r(z),$ $D_{-r}(z)=D_r(z),$ and $C_{-r}(z)=(C_r(z))^{-1}$ (\cite{EllingsrudGöttscheLehn} Theorem 5.3), so we restrict to the case $r\geq 0$. \\

For $r=0,1$, the coefficients of the Verlinde series are given by the combinatorially suggestive formulas,
\begin{equation}\label{r=01equation} \chi(X^{[n]},L_n) = {\chi(L)+n-1 \choose n}, \hspace{1cm}\text{and}\hspace{1cm} \chi(X^{[n]},L_n\otimes E) = {\chi(L) \choose n}, \end{equation}
valid for any smooth surface $X$ and line bundle $L$ (\cite{EllingsrudGöttscheLehn} Lemma 5.1). For $\chi(L)>0$, these formulas count the number of ways to choose $n$ objects from a set of $\chi(L)$ with and without repetitions respectively. In both cases $r=0,1$ we have $B_r=C_r=D_r=1$, and one can give formulas for $A_r$.\\

One approach to determine Verlinde series for $r>1$ is to focus on particular surfaces with additional structure. For example, the Hilbert scheme of points on a K3 surface is a symplectic manifold, and Ellingsrud, Göttsche, and Lehn use this additional structure to show that $\chi(X^{[n]},L_n\otimes E^r) = {\chi(L)-(r^2-1)(n-1) \choose n}$ for any K3 surface $X$, line bundle $L$, and integer $r$ (\cite{EllingsrudGöttscheLehn}\label{K3formula} Theorem 5.3). The enumerative data of the underlying K3 surface is $\chi(\Oo_X)=2$ and $K_X=0$, so this result is a formula for the coefficients of the Verlinde series $\V(z) = A_r(z)^{\chi(L)}\cdot B_r(z)^2$. From this, Ellingsrud, Göttsche, and Lehn extract formulas for $A_r$ and $B_r$ for any integer $r$. A key consequence of the structural formula (\ref{eqn:factorization}) is that these formulas for $A_r$ and $B_r$ deduced from the K3 case determine the Verlinde series for any surface $X$ for which $K_X=0$, since the unknown series $C_r$ and $D_r$ do not contribute to their Verlinde series.\\

Formulas for Verlinde series involving $C_r$ and $D_r$ have proved more difficult to find. Recently, using the theory of Macdonald polynomials, Göttsche and Mellit \cite{GM} gave a substantially more complicated formula for $C_r$ and a conjectural formula for $D_r$ for arbitrary $r$. In particular, their formula for $C_r$ determines the Verlinde series for any surface $X$ for which $K_X^2=0$, extending the known $K_X=0$ case.\\

Our first main result is a formula for the Euler characteristics of line bundles on the Hilbert scheme of points on $\P^1\times \P^1$. The relevant enumerative invariants of the pair $(X,L)=(\P^1\times\P^1,\Oo(d_1,d_2))$ are $\chi(L)=(d_1+1)(d_2+1)$, $\chi(\Oo_X)=1$, $c_1(L)\cdot K_X = -2(d_1+d_2)$, and $K_X^2=8$. This result is therefore an explicit formula for the coefficients of the Verlinde series
\begin{equation}\label{p1p1Verlinde} \V(z) = (A_r(z))^{(d_1+1)(d_2+1)}\cdot B_r(z) \cdot (C_r(z))^{-2(d_1+d_2)-4}\cdot(D_r(z))^{8} \end{equation}
for any integers $d_1,d_2$ and $r>0$.\\

To state the formula, we first set up some combinatorial notation. For any vector $\delta = (\delta_1,\dots,\delta_{n-1})\in \{0,1,\dots,r\}^{n-1}$, define the statistics $|\delta| = \sum_{i=1}^{n-1}\delta_i,$ $c(\delta)  = 1+ \# \{ i=1,\dots,n-1 \,|\, \delta_i\neq 0 \},$ and $\ell(\delta)  = 1+ \# \{ i=1,\dots,n-1 \,|\, \delta_i=r \}$. There are exactly $c(\delta)$ distinct numbers in the list $0,\delta_1,\delta_1+\delta_2,\dots,\delta_1+\cdots+\delta_{n-1}$ which we label in increasing order $a_1,\dots,a_c$ and we write $n_k = n_k(\delta)$ for the number of occurrences of $a_k$ in this list. Finally, for each $k=1,\dots,c$ we define $w_k(\delta)  = \sum_{i=1}^c n_i \max\{r-|a_k-a_i|,0\}.$

\begin{theorem}\label{P1P1}
    For $X=\P^1\times\P^1$, any line bundle $L=\Oo(d_1,d_2)$, and $r>0$,
    \[ \chi(X^{[n]},L_n\otimes E^r) =  \sum_{\delta\in \{0,1,\dots,r\}^{n-1}} {d_1 - |\delta|+\ell(\delta) \choose \ell(\delta)}\prod_{k=1}^{c(\delta)} {d_2- w_k(\delta)+r +n_k(\delta) \choose n_k(\delta)}. \]
\end{theorem}

Theorem \ref{P1P1} is proved in Section \ref{sec:Hirz}. We note that the quantity $\chi(X^{[n]},L_n\otimes E^r)$ is clearly symmetric in $d_1,d_2$, whereas the symmetry of the formula given in Theorem \ref{P1P1} is not at all apparent. The asymmetry of the expression comes from our study of a (symmetric) collection of polynomials using a non-symmetric term order (see Sections \ref{sec:c2} and \ref{sec:globalsections}). It would be interesting to show directly that the formula in Theorem \ref{P1P1} is symmetric in $d_1,d_2$.\\

Specializing Theorem \ref{P1P1} allows one to determine the series $C_r$ and $D_r$ explicitly. For example, when $(d_1,d_2)=(-1,-1)$ and $(d_1,d_2)=(-1,-2)$ the previous formula gives the coefficients of $\V(z) = B_r(z)\cdot D_r(z)^8$ and $\V(z) = B_r(z)\cdot C_r(z)^2 \cdot D_r(z)^8$ respectively. Since $B_r(z)$ is known, one can solve for $D_r(z)$ and $C_r(z)$ by dividing and taking roots of these power series. Theorem \ref{P1P1} therefore determines the Verlinde series $\V_{X,L,r}(z)$ for any surface $X$, line bundle $L$, and integer $r$. \\

The expressions for $C_r$ and $D_r$ extracted from Theorem \ref{P1P1} appear quite different from those given by Göttsche and Mellit \cite{GM}. In particular, it remains an open problem to show that their formula for $D_r$ is correct, and to show directly that the two different expressions for $C_r$ agree. Theorem \ref{P1P1} provides a new way to establish such conjectures: It suffices to show that a conjectured expression for $D_r$ gives the correct value of any one of the Verlinde series determined by Theorem \ref{P1P1}, e.g. $\V(z) = B_r(z)\cdot D_r(z)^8$.  \\

Finally, we note that the determination of Verlinde series by Theorem \ref{P1P1} also determines all higher rank Segre series by the Segre-Verlinde correspondence established by Göttsche and Mellit \cite{GM}. However, it is not clear how to extract closed form formulas for these Segre series from Theorem \ref{P1P1}. A first step in this direction might be to show that Theorem \ref{P1P1} is compatible under this correspondence with the formulas for Segre series established by Marian, Oprea, and Pandharipande \cite{MOPSegre} and Yao \cite{Y} in certain ranks.\\

We obtain Theorem \ref{P1P1} as a consequence of our second main result, a combinatorial interpretation for global sections of ample line bundles on Hilbert schemes of points on Hirzebruch surfaces. For now, we remain restricted to the special case $X=\P^1\times \P^1$. Consider the line bundle $L=\Oo(d_1,d_2)$ on $X=\P^1\times \P^1$. By the toric geometry of $X$ \cite{CLS}, there is a basis for the global sections $H^0(X,L)$ indexed by the integer points in the rectangle $P_L=[0,d_1]\times[0,d_2]\subseteq \R^2$. To generalize this to $X^{[n]}$, we introduce the following terminology. Let $(\bfa,\bfb)= (a_1,b_1),\dots,(a_n,b_n)$ be an ordered $n$-tuple of integer points in $P_L$. For any integer $r\geq 0$, we say that $(\bfa,\bfb)$ is \textit{$r$-lexicographically increasing in $P_L$} if:
\begin{enumerate}
    \item $a_i\leq a_{i+1}$ for each $i=1,\dots,n-1$,
    \item if $a_i=a_{i+1}$ then $b_{i+1}\geq b_i+r$ for each $i=1,\dots,n-1$, and
    \item $\sum_{i=1}^{j-1} \max\{ r-(a_j-a_i),0\} \leq b_j \leq d_2 - \sum_{k=j+1}^n \max\{ r-(a_k-a_j),0\}$ for each $j=1,\dots,n$.
\end{enumerate}
This final condition says that not only do we have $b_j\in [0,d_2]$ (because $(a_j,b_j)\in P_L$), but any other point in the tuple $(\bfa,\bfb)$ whose $a$-coordinate is close to $a_j$ imposes an additional constraint on $b_j$. One checks that an $n$-tuple of integer points $(\bfa,\bfb)$ in $P_L$ is $0$-lexicographically increasing in $P_L$ if and only if $(a_1,b_1)\leq \cdots \leq (a_n,b_n)$ in lexicographic order, and $1$-lexicographically increasing if and only if $(a_1,b_1)< \cdots < (a_n,b_n)$ in lexicographic order. For larger $r$, $r$-lexicographic increasingness in $P_L$ is a stronger notion of separatedness for $n$-tuples.

\begin{example}
    Let $P=[0,2]\times [0,4]$, and $n=r=3$. There are ten $3$-lexicographically increasing triples of points in $P$. These triples are depicted below where the points in each triple are in increasing lexicographic order. For example, the first diagram represents the triple $(a_1,b_1)=(0,0)$, $(a_2,b_2)=(0,3)$, and $(a_3,b_3)=(2,2)$.
    
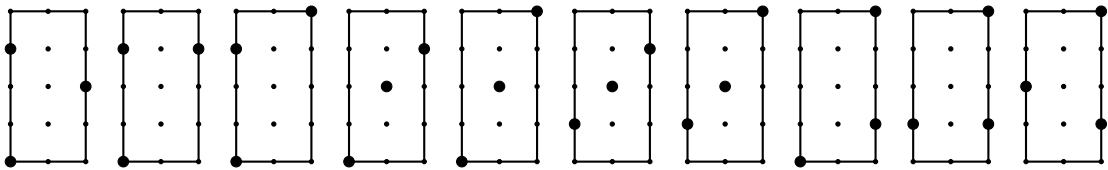
\[\begin{tikzpicture}
    
    \foreach \i in {0,...,9}{
        \draw[thick] (1.5*\i,0) -- (1.5*\i+1,0) -- (1.5*\i+1,2) -- (1.5*\i,2) -- cycle;
        \foreach \j in {0,...,2}
        \foreach \k in {0,...,4}
            \filldraw (1.5*\i+0.5*\j,0.5*\k) circle (.8pt);
    }
    \filldraw (0,0) circle (2pt);
    \filldraw (0,1.5) circle (2pt);
    \filldraw (1,1) circle (2pt);    
    
    \filldraw (1.5,0) circle (2pt);
    \filldraw (1.5,1.5) circle (2pt);
    \filldraw (2.5,1.5) circle (2pt);
    
    \filldraw (3,0) circle (2pt);
    \filldraw (3,1.5) circle (2pt);
    \filldraw (4,2) circle (2pt);
    
    \filldraw (4.5,0) circle (2pt);
    \filldraw (5,1) circle (2pt);
    \filldraw (5.5,1.5) circle (2pt);
    
    \filldraw (6,0) circle (2pt);
    \filldraw (6.5,1) circle (2pt);
    \filldraw (7,2) circle (2pt);
    
    \filldraw (7.5+3,0) circle (2pt);
    \filldraw (8.5+3,0.5) circle (2pt);
    \filldraw (8.5+3,2) circle (2pt);
    
    \filldraw (9-1.5,0.5) circle (2pt);
    \filldraw (9.5-1.5,1) circle (2pt);
    \filldraw (10-1.5,1.5) circle (2pt);
    
    \filldraw (10.5-1.5,0.5) circle (2pt);
    \filldraw (11-1.5,1) circle (2pt);
    \filldraw (11.5-1.5,2) circle (2pt);
        
    \filldraw (12,0.5) circle (2pt);
    \filldraw (13,0.5) circle (2pt);
    \filldraw (13,2) circle (2pt);
    
    \filldraw (13.5,1) circle (2pt);
    \filldraw (14.5,0.5) circle (2pt);
    \filldraw (14.5,2) circle (2pt);
    
    \end{tikzpicture}\]
\end{example}

In Section \ref{sec:Hirz} we extend the definition of $r$-lexicographically increasing $n$-tuples of points in rectangles to trapezoids corresponding to line bundles on Hirzebruch surfaces. Our second main result uses this notion to extend the toric combinatorics of Hirzebruch surfaces to their Hilbert scheme of points.

\begin{theorem}\label{introtheorem}
    Let $X$ be a Hirzebruch surface and $L_n\otimes E^r$ an ample line bundle on $X^{[n]}$ (so in particular $r>0$). The set of $r$-lexicographically increasing $n$-tuples of points in $P_L$ indexes a basis of $H^0(X^{[n]},L_n\otimes E^r).$ In this case, $\chi(X^{[n]},L_n\otimes E^r)$ is equal to the number of $r$-lexicographically increasing $n$-tuples of points in $P_L$.
\end{theorem}

The statement that $\chi(X^{[n]},L_n\otimes E^r)$ coincides with the number of $r$-lexicographically increasing $n$-tuples of points in $P_L$ is proved in Theorem \ref{Thm:Eulerchar}. The Frobenius splitting of $X^{[n]}$ implies that $\chi(X^{[n]},L_n\otimes E^r)=\dim H^0(X^{[n]},L_n\otimes E^r)$ for any ample $L_n\otimes E^r$ \cite{KT}, so abstractly we know that any basis of $H^0(X^{[n]},L_n\otimes E^r)$ has as many elements as the number of such $n$-tuples. In Section \ref{sec:globalsections}, however, we give a much more direct sense in which these $n$-tuples correspond to global sections. The global sections can be naturally identified with a certain $\C$-linear span of polynomials $W\subseteq \C[x_1,y_1,\dots,x_n,y_n]$, and the $r$-lexicographically increasing $n$-tuples $(\bfa,\bfb)$ in $P_L$ are precisely the leading term exponents of polynomials $f = x_1^{a_1}y_1^{b_1}\cdots x_n^{a_n}y_n^{b_n}+\cdots$ in $W$ with respect to a certain term order.\\

The proof of Theorem \ref{introtheorem} uses the (standard) fact that the Verlinde series for toric surfaces can be expressed in terms of equivariant Verlinde series for $\C^2$, which we review in Section \ref{sec:torictoaffine}. Equivariant Verlinde series for $\C^2$ can then be computed using the equivariant localization formula, a strategy that was used to compute several coefficients for $C_r$ and $D_r$ in \cite{EllingsrudGöttscheLehn}. However, the combinatorics of this expression are unwieldy. It is not even clear from these expressions that the Euler characteristics of line bundles on $X^{[n]}$ should be integers.\\

Our present results are based on a new combinatorial interpretation of the equivariant Verlinde series for $\C^2$, given in Section \ref{sec:c2}. We interpret the coefficients as generating functions of integer points in certain convex sets, or equivalently of certain $n$-tuples of integer points in the plane. This reduces Theorem \ref{introtheorem} to an identity of generating functions of integer points in certain convex sets, and we show in Section \ref{sec:Hirz} that this identity can be deduced from Brion's formula \cite{Brion1988}.\\

\textit{Acknowledgements:} I thank Rahul Pandharipande and Anton Mellit for helpful comments on a preliminary version of this paper, and Dave Anderson for valuable discussions during the early stages of this project.

\section{Background on Verlinde Series and Toric Surfaces}\label{sec:torictoaffine}

Let $X$ be a smooth, projective, toric surface equipped with an action of $T\simeq (\C^*)^2$. For background on the theory of toric varieties we refer to \cite{CLS}. The same torus $T$ acts on the Hilbert scheme $X^{[n]}$ by pull-back of subschemes. For any line bundle $L$ on $X$, the symmetric line bundle $L\boxtimes \cdots\boxtimes L$ on $X^n$ descends to the symmetric power $X^n/S_n$. Let $L_n$ denote the pullback of this line bundle to the Hilbert scheme $X^{[n]}$ via the Hilbert-Chow morphism $X^{[n]}\to X^n/S_n$. When $L$ is a $T$-equivariant line bundle, $L_n$ inherits a $T$-equivariant structure as well.\\

Let $\Sigma_n\subseteq X^{[n]}\times X$ denote the universal family, and $\Oo_n$ its structure sheaf. The line bundle $E$ on $X^{[n]}$ is defined as $E=\det(q_*(\Oo_n\otimes p^*\Oo_X))$ where $q$ and $p$ are the natural projections of $X^{[n]}\times X$ onto $X^{[n]}$ and $X$ respectively. In terms of divisor classes, we have $c_1(E)=-\frac{1}{2}[D]$ where $D\subseteq X^{[n]}$ is the exceptional divisor of the Hilbert-Chow morphism parametrizing reduced length-$n$ subschemes of $X$. The line bundle $E$ is also equipped with a natural $T$-equivariant structure. A fundamental theorem due to Fogarty \cite{FogartyII} states that $\Pic(X^{[n]})\simeq \Pic(X)\times \Z E$, so that every line bundle on $X^{[n]}$ is of the form $L_n\otimes E^r$ for some line bundle $L$ on $X$ and $r\in \Z$.\\

Given a $T$-equivariant line bundle $L$ on $X$ and an integer $r$, the equivariant Euler characteristic of the line bundle $L_n\otimes E^r$ is defined as 
\[ \chi^T\left(X^{[n]},L_n\otimes E^r\right) = \sum_{i=0}^{2n}\sum_{(a,b)\in \Z^2} (-1)^i\, t^aq^b \, \dim_\C H^i\left(X^{[n]},L_n\otimes E^r\right)_{(a,b)}, \]
where $V_{(a,b)}$ denotes the $(a,b)$-weight space of a vector space $V$ on which $T$ acts. In other words, $V_{(a,b)}$ is the set of all $v\in V$ such that $(t,q)\cdot v = t^aq^b v$ for all $(t,q)\in T$. \\

Since $X^{[n]}$ is projective, the cohomology groups of line bundles on $X^{[n]}$ are finite-dimensional and so their equivariant Euler characteristics are Laurent polynomials in $t$ and $q$. These can be assembled into equivariant refinements of the Verlinde series introduced in the introduction,
\[ \V^T_{X,L,r}(z) = \sum_{n= 0}^\infty z^n\cdot \chi^T(X^{[n]},L_n\otimes E^r) \in \Z[t^{\pm 1},q^{\pm1}][[z]]. \]
Unlike ordinary Velinde series, these equivariant series can be defined for non-projective toric surfaces. In particular, consider the action of $T$ on $\C^2$ defined on the coordinate ring $\C[x,y]$ by $(t,q)\cdot x=t x$ and $(t,q)\cdot y=q y$. We equip $\C^2$ with a $T$-linearized line bundle $L$. Although any such line bundle is trivial, it may be equipped with a non-trivial torus action and it will be useful for us to keep track of this data. In this case the equivariant Verlinde series is defined as
\[ \V^T_{\C^2,L,r}(z) = \sum_{n= 0}^\infty z^n\cdot \chi^T((\C^2)^{[n]},L_n\otimes E^r), \]
where the coefficients are now formal power series in $t,q$ rather than Laurent polynomials. In fact, these coefficients can be represented by rational functions in $t,q$. Suppose that $L$ is equipped with the $T$-action by the character $t^{m_1}q^{m_2}$. The equivariant Euler characteristic of $L_n\otimes E^r$ can be computed by the equivariant localization formula \cite{Haiman1998tQN},
\begin{equation}\label{C2localization}
    \chi^T((\C^2)^{[n]},L_n\otimes E^r) = t^{nm_1}q^{nm_2} \sum_{\lambda \vdash n} \frac{t^{r n(\lambda)}q^{r n(\lambda')}}{\prod_{e\in \lambda}(1-t^{1+l(e)}q^{-a(e)})(1-t^{-l(e)}q^{1+a(e)})}.
\end{equation}
The above sum is over partitions $\lambda$ of $n$, $a(e)$ and $l(e)$ denote the arm and leg lengths respectively of a cell $e$ in $\lambda$, $\lambda'$ denotes the conjugate partition to $\lambda$, and $n(\mu)=\sum_{e\in \mu} l(e)$. More generally, for an action of $T$ on $\C^2$ by distinct characters $t^{u}q^{v}$ and $t^{u'}q^{v'}$ rather than $t$ and $q$, one replaces $t$ and $q$ in the sum above with the corresponding characters.\\

The following standard result expresses equivariant Verlinde series for projective toric surfaces in terms of those for $\C^2$.

\begin{proposition}\label{C2reduction}
    For any smooth, projective, toric surface $X$ equipped with a $T$-linearized line bundle $L$ and integer $r$,
    \[ \V^T_{X,L,r}(z) = \prod_{i}\V^T_{U_i,L|_{U_i},r}(z), \]
    where the product is over the open sets $\C^2\simeq U_i\subseteq X$ in the standard $T$-invariant affine open cover of $X$.
\end{proposition}

\begin{proof}
    Let $U_1,\dots,U_k$ denote the standard affine open cover of $X$, with $p_1,\dots,p_k$ the corresponding torus fixed points. The generating function identity in the proposition is equivalent to the expression
    \[ \chi^T(X^{[n]},L_n\otimes E^r) = \sum_{n_1+\cdots+n_k=n} \prod_{i=1}^k \chi^T(U_i^{[n_i]},(L|_{U_i})_{n_i}\otimes E^r). \]
    To see that these coincide, expand each equivaraint Euler characteristic using the localization formula as in \cite{Haiman1998tQN}. The fixed points $\xi\in (X^{[n]})^T$ can be decomposed as $\xi=\xi_1\sqcup \cdots\sqcup \xi_k$ where $\xi_i\subseteq U_i$ is a $T$-fixed scheme supported at $p_i$. The term corresponding to $\xi$ in the localization expression for $\chi^T(X^{[n]},L_n\otimes E^r)$ is the product of the terms corresponding to $\xi_1,\dots,\xi_k$ in $\prod_{i=1}^k \chi^T(U_i^{[n_i]},(L|_{U_i})_{n_i}\otimes E^r)$ where $n_i$ denotes the length of the component $\xi_i$.
\end{proof}

In Section \ref{sec:Hirz} we study the case $X=\mathcal H_s$ a Hirzebruch surface, defined as the projectivization of the split rank two vector bundle $\Oo_{\P^1}\oplus\Oo_{\P^1}(s)$ over $\P^1$ for some fixed integer $s\geq 0$. Let $D_1$ be the class of the projectivized $T$-fixed section with self-intersection $s$, and $D_2$ be the class of a fiber. The line bundles associated to $D_1$ and $D_2$ freely generate the Picard group of $X$. We identify the Newton polygon of the line bundle $L = \Oo_{X}(d_1 D_1+d_2 D_2)$ with the polygon in $\R^2$ defined by $0\leq x \leq d_1$ and $0\leq y \leq d_2+ sd_1$ as depicted below, and denote this polygon by $P_{L}\subseteq \R^2$.\\

\vspace{-2ex}
\begin{figure}[h]
    \centering
    \[ \begin{tikzpicture}
        \def\a{1.7};
        \def\b{1.7};
        \def\r{.1};
        \def\l{.6};
        \draw (0,0) -- (\a,0) -- (\a,\a+\b) -- (0,\b) -- cycle;
        \filldraw (0,0) circle (2pt);
        \filldraw (\a,0) circle (2pt);
        \filldraw (\a,\a+\b) circle (2pt);
        \filldraw (0,\b) circle (2pt);
        \draw[->] (\r,-\r) -- node [midway, below] {$t$}  (\r+\l,-\r);
        \draw[->] (-\r,\r) -- node [midway, left] {$q$}  (-\r,\l+\r);
        \draw[->] (\a-\r,-\r) -- node [midway, below] {$t^{-1}$}  (\a-\r-\l,-\r);
        \draw[->] (\a+\r,\r) -- node [midway, right] {$q$}  (\a+\r,\r+\l);
        \draw[->] (\a-1.4*\r,\a+\b) -- node [midway, above] {$q^{-1}t^{-s}\hspace{3ex}$}  (\a-1.4*\r-\l,\a+\b-\l);
        \draw[->] (\a+\r,\a+\b-\r) -- node [midway, right] {$q^{-1}$}  (\a+\r,\a+\b-\r-\l);
        \draw[->] (0,\b+1.4*\r) -- node [midway, above] {$qt^s\hspace{1ex}$}  (\l,\b+\l+1.4*\r);
        \draw[->] (-\r,\b-\r) -- node [midway, left] {$q^{-1}$}  (-\r,\b-\r-\l);
        \node[below left] at (0,0) {$1$};
        \node[below right] at (\a,0) {$t^{d_1}$};
        \node[above right] at (\a,\a+\b) {$t^{d_1} q^{d_2+s d_1}$};
        \node[above left] at (0,\b) {$q^{d_2}$};
        \end{tikzpicture} \]
    \label{fig:polygon}
    \vspace{-1cm}
    \caption{The polygon $P_L$ corresponding to $L=\Oo_{X}(d_1 D_1+d_2 D_2)$. Each vertex corresponds to a fixed point $p_i$ and is labeled with the character of $L|_{U_{p_i}}$. The rays extending from each vertex are labelled with the characters by which $T$ acts on $U_i$.}
\end{figure}
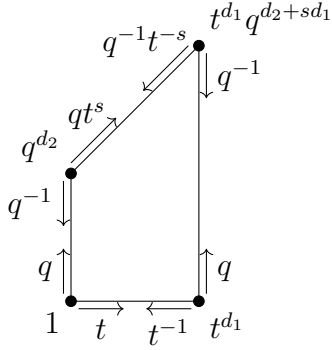

There are four standard affine opens $U_1,\dots,U_4$ with corresponding fixed points $p_1,\dots,p_4$ corresponding to the vertices $(0,0), (d_1,0), (d_1,d_2+sd_1),$ and $(0,d_2)$ of $P_L$ respectively. For readability, we sometimes use the notation $U_{\corner{}}=U_1$, $U_{\corner{rotate=90}}=U_2,$ $U_{\corner{rotate=180}}=U_3$, and $U_{\corner{rotate=270}}=U_4$.\\

We introduce notation for the rational function $\sigma_{n,r}(t,q) = \chi^T((\C^2)^{[n]},E^r)$ given by (\ref{C2localization}). For the Hirzebruch surface $X=\mathcal H_s$ and line bundle $L=\Oo_{X}(d_1 D_1+d_2 D_2)$, we will need four variants of this rational function corresponding to the open sets $U_1,\dots,U_4\subseteq X$:
\begin{align}\label{sigmais}
\begin{split}
    \sigma^{(1)}_{n,r}(t,q) = \sigma^{\corner{}}_{n,r}(t,q) & := \chi^T(U_1^{[n]},(L|_{U_1})_n\otimes E^r) = \sigma_{n,r}(t,q) \\
    \sigma^{(2)}_{n,r}(t,q) =\sigma^{\corner{rotate=90}}_{n,r}(t,q) & := \chi^T(U_2^{[n]},(L|_{U_2})_n\otimes E^r) = t^{n d_1}\sigma_{n,r}(t^{-1},q) \\
    \sigma^{(3)}_{n,r}(t,q) =\sigma^{\corner{rotate=180}}_{n,r}(t,q) & := \chi^T(U_3^{[n]},(L|_{U_3})_n\otimes E^r) = t^{n d_1}q^{n(d_2+sd_2)}\sigma_{n,r}(q^{-1}t^{-s},q^{-1}) \\
    \sigma^{(4)}_{n,r}(t,q) =\sigma^{\corner{rotate=270}}_{n,r}(t,q) & := \chi^T(U_4^{[n]},(L|_{U_4})_n\otimes E^r) = q^{n d_2}\sigma_{n,r}(qt^s,q^{-1}) .
\end{split}
\end{align}
With this notation, the equality between coefficients on $z^n$ in Proposition \ref{C2reduction} gives the identity
\begin{equation}\label{Hirzidentity}
    \chi^T(X^{[n]},L_n\otimes E^r) = \sum_{n_1+\cdots+n_4=n} \left(\prod_{i=1}^4 \sigma^{(i)}_{n_i,r}(t,q)\right).
\end{equation}

\section{Combinatorial Verlinde Series for the Affine Plane}\label{sec:c2}

In this section, we give a new formula for the rational function $\sigma_{n,r}(t,q)=\chi^T((\C^2)^{[n]},E^r)$ in the case $r>0$. First, we recall an explicit construction of $(\C^2)^{[n]}$ following Haiman \cite{Haiman1998tQN}. \\

Consider the action of the symmetric group $S_n$ on $\C[\mathbf{x,y}] = \C[x_1,y_1,\dots,x_n,y_n]$ defined by $\sigma\cdot x_i = x_{\sigma(i)}$ and $\sigma\cdot y_i = y_{\sigma(i)}$. A polynomial $f\in \C[\mathbf{x,y}]$ is said to be symmetric if $\sigma\cdot f=f$ for all $\sigma\in S_n$, and alternating if $\sigma\cdot f = \mathrm{sgn}(\sigma)f$ for all $\sigma\in S_n$.\\

Let $A^0\subseteq \C[\mathbf{x,y}]$ denote the space of all symmetric polynomials, and $A^1\subseteq \C[\mathbf{x,y}]$ the space of all alternating polynomials. Each $n$-tuple of not-necessarily distinct points $(a_1,b_1),\dots,(a_n,b_n)\in \Z^2_{\geq 0}$ corresponds to a monomial symmetric polynomial $x_1^{a_1}y_1^{b_1}\cdots x_n^{a_n}y_n^{b_n}+$ (symmetric terms). Similarly, each $n$-tuple of distinct points $(a_1,b_1),\dots,(a_n,b_n)\in \Z^2_{\geq 0}$ corresponds to an alternating polynomial $\det(x_i^{a_j}y_i^{b_j})_{ij}.$ These polynomials form bases of $A^0$ and $A^1$ as $(a_1,b_1),\dots,(a_n,b_n)$ vary over all such $n$-tuples up to reordering.\\

For $r>1$, let $A^r\subseteq \C[\mathbf{x,y}]$ denote the span of products $f_1\cdots f_r$ where $f_1,\dots,f_r\in A^1$, so that $R=A^0\oplus A^1\oplus A^2\oplus \cdots$ forms a graded ring. Haiman shows that the Hilbert scheme of $n$ points in $\C^2$ is isomorphic to $\Proj(R)$ in such a way that the natural morphism $\Proj(R)\to \Spec(A^0)$ corresponds to the Hilbert-Chow morphism (\cite{Haiman1998tQN}, Proposition 2.6).\\

In the notation of the previous section, the line bundle $E$ corresponds to $\Oo(1)$ under the identification of $(\C^2)^{[n]}$ with $\Proj(R)$. Haiman's results can be used to show that the ring $R$ is integrally closed, and therefore for each $r\geq 0$ there is an isomorphism (see \cite{Cavey}, Corollary 3.10)
\[ H^0((\C^2)^{[n]}, E^r) \simeq A^r. \]

In contrast to the spaces of symmetric and alternating polynomials, it is unclear how to obtain a basis of $A^r$ when $r>1$. Such a basis can be extracted from our earlier results \cite{Cavey}, which we now summarize.\\

Let $(\bfa,\bfb)= (a_1,b_1,\dots,a_n,b_n)\in \R^{2n}$, and regard a point $(\bfa,\bfb)\in\R^{2n}$ as an ordered $n$-tuple of points $(a_1,b_1),\dots,(a_n,b_n)\in \R^2$. Define $P_n\subseteq \R^{2n}$ to be the convex hull of the set of nonnegative integer vectors $(\bfa,\bfb)\in \Z^{2n}_{\geq 0}$ such that $(a_1,b_1)<\cdots<(a_n,b_n)$ in lexicographic order. In \cite{Cavey} we showed by induction on $n$ that $P_n$ is given explicitly by
\[ P_n  = \left\{\begin{varwidth}{10 in} \setstretch{1.5}
$(\bfa,\bfb)\in \R^{2n}$
 \end{varwidth} 
 \hspace{1ex}\left|\hspace{1ex} 
 \begin{varwidth}{10 in}
$0\leq a_1\leq a_2\leq\cdots\leq a_n$, \\
for each $j=1,\dots,n-1$, if $a_j = a_{j+1}$ then $b_{j+1}\geq b_j+1$, and\\
$b_j\geq \sum_{i=1}^{j-1}\max\{1-(a_j-a_i),0\} $ for all $1 \leq j \leq n$
 \end{varwidth} 
 \right.
 \right\}
 \setstretch{1}. \]

We equip $\C[\mathbf{x,y}]$ with the lexicographic term order where the variables are ordered as $x_1>\cdots>x_n>y_1>\cdots>y_n$. The following result allows us to extract a basis of $A^r$ for $r>1$.

\begin{theorem}\cite{Cavey} \label{C2}
    For each $r>0$, the integer points in the $r$-fold dilation $(\bfa,\bfb)\in (rP_n) \cap \Z^{2n}$ are precisely the vectors that appear as exponents of trailing terms $x_1^{a_1}\cdots x_n^{a_n}y_1^{b_1}\cdots y_n^{b_n}$ of polynomials in $A^r$.
\end{theorem}

The dilation $rP_n$ can be described explicitly by 

\[ rP_n  = \left\{\begin{varwidth}{10 in} \setstretch{1.5}
$(\bfa,\bfb)\in \R^{2n}$
 \end{varwidth} 
 \hspace{1ex}\left|\hspace{1ex} 
 \begin{varwidth}{10 in}
$0\leq a_1\leq a_2\leq\cdots\leq a_n$, \\
for each $j=1,\dots,n-1$, if $a_j = a_{j+1}$ then $b_{j+1}\geq b_j+r$, and\\
$b_j\geq \sum_{i=1}^{j-1}\max\{r-(a_j-a_i),0\} $ for all $1 \leq j \leq n$
 \end{varwidth} 
 \right.
 \right\}
 \setstretch{1}. \]

We call any integer vector $(\bfa,\bfb)\in rP_n$ an $r$\textit{-lexicographically increasing $n$-tuple in $\R^2_{\geq0}$}, as these are precisely the $n$-tuples that can be written as coordinate-wise sums of $r$ $n$-tuples of distinct points in $\Z^2_{\geq0}$ written in increasing lexicographical order.

\begin{corollary}\label{C2convexgeometry}
    For any $r>0$, the rational function $\sigma_{n,r}(t,q) = \chi^T((\C^2)^{[n]},E^r) $ is given by
    \[ \sigma_{n,r}(t,q)= \sum_{(\mathbf{a,b})\in (rP_n)\cap \Z^{2n}} t^{a_1+\cdots+a_n}q^{b_1+\cdots+b_n} \]
\end{corollary}

\begin{proof}
    By the Frobenius splitting of $(\C^2)^{[n]}$ \cite{KT}, the higher cohomology groups of $E^r$ vanish for all $r>0$, and so the coefficient on $t^aq^b$ in the series $\chi^T((\C^2)^{[n]},E^r)$ is equal to the dimension of $H^0((\C^2)^{[n]},E^r)_{(a,b)}$. Under the identification $H^0((\C^2)^{[n]},E^r)\simeq A^r$ described above, the $T$-action on $H^0((\C^2)^{[n]},E^r)$ corresponds to the action on $A^r$ given by $(t,q)\cdot x_i = tx_i$ and $(t,q)\cdot y_i=qy_i$ for all $i=1,\dots,n$ and $(t,q)\in T$. The weight space $A^r_{(a,b)}$ therefore consists of those polynomials $f\in A^r$ such that $(a_1+\cdots+a_n,b_1+\cdots+b_n)=(a,b)$ for every term $x_1^{a_1}\cdots x_n^{a_n}y_1^{mb_1}\cdots y_n^{b_n}$ of $f$.\\
    
    By Theorem \ref{C2}, the integer points $(\mathbf{a,b})\in rP_n\cap \Z^2$ are the lexicographic trailing term exponents of polynomials in $A^r$. The integer points $(\mathbf{a,b})\in rP_n\cap \Z^{2n}$ corresponding to trailing terms of polynomials in the weight space $A^r_{(a,b)}$ are those with $(a_1+\cdots+a_n,b_1+\cdots+b_n)=(a,b)$. Any collection of polynomials in $A^r_{(a,b)}$ with pairwise distinct trailing terms is linearly independent, and any additional polynomial not in their linear span can be reduced modulo the collection to obtain a new trailing term. This implies that the number of $(\bfa,\bfb)\in rP_n\cap \Z^{2n}$ with $(a_1+\cdots+a_n,b_1+\cdots+b_n)=(a,b)$, corresponding to all the trailing terms of polynomials in $A^r_{(a,b)}$, is equal to the dimension of $A^r_{(a,b)}$, which completes the proof.
\end{proof}

\begin{example}
    Let $n=2$ and $r=2$. By formula (\ref{C2localization}), we have
    \[ \chi^T((\C^2)^{[2]},E^2) = \frac{t^2+tq+q^2-t^2q^2}{(t^2-1)(t-1)(q^2-1)(q-1)}. \]
    One can compute, for example, that the coefficient on the $t^3q^2$ term of the power series expansion at $t=q=0$ of this rational function is $5$. Correspondingly, there are 5 integer pairs $((a_1,b_1),(a_2,b_2))$ corresponding to points $(a_1,b_1,a_2,b_2)\in 2P_2$ with $a_1+a_2=3$ and $b_1+b_2=2$. They are the pairs $((0,0),(3,2))$, $((0,1),(3,1))$, $((0,2),(3,0))$, $((1,0),(2,2))$, and $((1,1),(2,1))$.
\end{example}

\section{Combinatorial Verlinde Series for Hirzebruch Surfaces}\label{sec:Hirz}

In this section, we use the combinatorial interpretation of $\sigma_{n,r}(t,q)$ given in Corollary \ref{C2convexgeometry} to study line bundles on the Hilbert scheme of points on a Hirzebruch surface. Fix the Hirzebruch surface $X=\mathcal H_s$ and line bundle $L=\Oo(d_1D_1+d_2D_2)$ with corresponding polytope $P_L$ as defined in Section \ref{sec:torictoaffine}, and fix an integer $r>0$. \\

Our basic collection of $n$-tuples will be the integer points in the set
\[
P_{n,r}^\circ = \left\{\begin{varwidth}{10 in} \setstretch{1.5}
$(\bfa,\bfb)\in \R^{2n}$
 \end{varwidth} 
 \hspace{1ex}\left|\hspace{1ex} 
 \begin{varwidth}{10 in}
$a_1\leq a_2\leq\cdots\leq a_n$, and \\
for each $j=1,\dots,n-1$, if $a_j = a_{j+1}$ then $b_{j+1}\geq b_j+r$
 \end{varwidth} 
 \right.
 \right\}
 \setstretch{1}. \]
This is the $r$-fold dilation of the convex hull of the set of all $n$-tuples of integer points $(\bfa,\bfb)$ such that $(a_1,b_1)<\cdots<(a_n,b_n)$ in lexicographic order. Next, we introduce four constraints on vectors $(\bfa,\bfb)\in P_{n,r}^\circ$ depending on $r$ and corresponding to the left, right, bottom, and top edges of $P_L$ respectively (as drawn in Figure \ref{fig:polygon}):
\begin{align*}
    \text{(Left)} \hspace{1cm}& 0  \leq a_1, \\
    \text{(Right)} \hspace{1cm}& a_n \leq d_1, \\
    \text{(Bottom)} \hspace{1cm}& b_j  \geq \sum_{i=1}^{j-1}\max\{r-(a_j-a_i),0\} \text{ for all } 1 \leq j \leq n, \text{ and}, \\
    \text{(Top)} \hspace{1cm}& b_j  \leq d_2+s a_j-\sum_{k=j+1}^{n}\max\{r-(a_k-a_j),0\} \text{ for all } 1 \leq j \leq n.
\end{align*}

We will need the following sets of $n$-tuples corresponding to the four vertex cones of $P_L$ and the polygon $P_L$ itself:
\begin{align}\label{ConeSetsDef}
\begin{split}
P^{(1)}_{n,r} &= P^{\corner{}}_{n,r} := \left\{ (\bfa,\bfb)\in P_{n,r}^\circ \, | \, (\bfa,\bfb) \text{ satisfies the bottom and left constraints} \right\}\\
P^{(2)}_{n,r} &= P^{\corner{rotate=90}}_{n,r} := \left\{ (\bfa,\bfb)\in P_{n,r}^\circ \, | \, (\bfa,\bfb) \text{ satisfies the bottom and right constraints} \right\}\\
P^{(3)}_{n,r} &= P^{\corner{rotate=180}}_{n,r} := \left\{ (\bfa,\bfb)\in P_{n,r}^\circ \, | \, (\bfa,\bfb) \text{ satisfies the top and right constraints} \right\}\\
P^{(4)}_{n,r} &= P^{\corner{rotate=270}}_{n,r} := \left\{ (\bfa,\bfb)\in P_{n,r}^\circ \, | \, (\bfa,\bfb) \text{ satisfies the top and left constraints} \right\}\\
P^{\square{}}_{n,r} & := \left\{ (\bfa,\bfb)\in P_{n,r}^\circ \, | \, (\bfa,\bfb) \text{ satisfies the top, bottom, left, and right constraints} \right\}
\end{split}
\end{align}

By the inequalities of $P_L$ given in Section \ref{sec:torictoaffine}, any $n$-tuple $(\bfa,\bfb)\in P^{\square{}}_{n,r}$ has $(a_1,b_1),\dots,(a_n,b_n)\in P_L$ since in particular $0\leq a_j\leq d_1$ and $0\leq b_j\leq d_2+sa_j$ for all $j=1,\dots,n$. The integer points $(\bfa,\bfb)\in P^{\square{}}_{n,r}\cap \Z^{2n}$ are the \textit{$r$-lexicographically increasing $n$-tuples in $P_L$}, as described in the introduction. As a consequence of Corollary \ref{C2convexgeometry}, the integer points $P^{(i)}_{n,r}$ give a combinatorial interpretation of the series $\sigma_{n,r}^{(i)}(t,q)$.

\begin{corollary}\label{combsigmas}
    For any $n,r>0$ and $i=1,\dots,4$ we have 
    \[ \chi^T(U_i^{[n]},(L|_{U_i})_n\otimes E^r) = \sigma^{(i)}_{n,r}(t,q) = \sum_{(\mathbf{a,b})\in P^{(i)}_{n,r}\cap \Z^{2n}} t^{a_1+\cdots+a_n}q^{b_1+\cdots+b_n}. \]
\end{corollary}

Our main object of study is the generating function of $r$-lexicographically increasing $n$-tuples in $P_L$,
\[ \sigma^{\square{}}_{n,r}(t,q):= \sum_{(\mathbf{a,b})\in P^{\square{}}_{n,r}\cap \Z^{2n}} t^{a_1+\cdots+a_n}q^{b_1+\cdots+b_n}. \]
The goal of this section is to establish the following relationship between this generating function and the Euler characteristic of $L_n\otimes E^r$.

\begin{theorem}\label{Thm:Eulerchar}
    Let $X=\mathcal H_s$, $L=\Oo(d_1D_1+d_2D_2)$ and $r>0$ as above. If $d_1,d_2>r(n-1)$, then $\sigma^{\square{}}_{n,r}(t,q) = \chi^T(X^{[n]},L_n\otimes E^r). $
\end{theorem}

The conditions $r>0$, and $d_1,d_2>r(n-1)$ exactly describe the set of ample line bundles on $X^{[n]}$ \cite{BertramCoskun}.\\

The idea of the proof is as follows: By (\ref{Hirzidentity}) we have an expression for $\chi^T(X^{[n]},L_n\otimes E^r)$ in terms of the rational functions $\sigma^{(i)}_{n,r}(t,q)$, and by Corollary \ref{combsigmas} these rational functions are the generating functions of integer points in the convex sets $P^{(i)}_{n,r}.$ We will show that the generating function $\sigma^{\square{}}_{n,r}(t,q)$ of integer points in $P^{\square{}}_{n,r}$ satisfies the same identity (\ref{Hirzidentity}) in terms of the $\sigma^{(i)}_{n,r}(t,q)$'s by expanding the sum using Brion's formula \cite{Brion1988}. To carry this plan out, we first need a relatively detailed study of the combinatorics of these $n$-tuples.\\

Given an integer point $(\bfa,\bfb) \in P^{\circ}_{n,r}\cap \Z^{2n}$ which we think of as an $n$-tuple of points $p_j=(a_j,b_j)$, let $\delta_i = \delta_i(\bfa,\bfb) = \max\{a_{i+1}-a_i,r\} $ for each $i=1,\dots,n-1$. We refer to the vector $\delta = (\delta_1,\dots,\delta_{n-1})\in \{0,1,\dots,r\}^{n-1}$ as the type of the $n$-tuple $(\bfa,\bfb).$ For a fixed element $\delta=(\delta_1,\dots,\delta_{n-1})\in \{0,1,\dots,r\}^{n-1}$, we define two partitions of each $n$-tuple $\{p_1,\dots,p_n\}$ of type $\delta$. A \textbf{block} is a maximal collection of successive points $\{p_i,p_{i+1},\dots,p_j\}$ such that $\delta_i,\dots,\delta_{j-1}<r$. A \textbf{column} is a maximal collection of successive points $\{p_i,p_{i+1},\dots,p_j\}$ such that $\delta_i,\dots,\delta_{j-1}=0.$ Clearly, each block is made up of a union of consecutive columns. We define $I = \{ i_1,\dots,i_\ell \} \subseteq \{1,\dots,n\}$ to be the indices of the first points in each block for any $n$-tuple of type $\delta$. In terms of $\delta$, this means $i_1=1$ and $i_{k+1}$ is the index at which the $k$th entry equal to $r$ appears in $\delta$. Finally, for each column $C = \{p_j,\dots,p_{j'}\}$ we define the statistics \[ L_C = \sum_{i=1}^{j-1}\max\{r-(\delta_i+\cdots+\delta_{j-1}),0 \}, \hspace{.5cm} \text{and}\hspace{.5cm} R_{C} = \sum_{k=j'+1}^n\max\{r-(\delta_{j'}+\cdots+\delta_{k-1}),0\}. \]

\begin{lemma}\label{typedeltainequalities}
    The set of integer points $(\bfa,\bfb)\in P^{\square{}}_{n,r}\cap \Z^{2n}$ of type $\delta$ are exactly the integer points satisfying the conditions
    \begin{enumerate}
        \item $0\leq a_{i_1}, \hspace{.5cm} a_{i_k}+\delta_{i_k}+\cdots+\delta_{i_{k+1}-1}\leq a_{i_{k+1}} \text{ for each } k=1,\dots,\ell-1, \text{ and}\hspace{.5cm} a_{i_\ell}+\delta_{i_\ell}+\cdots+\delta_{n-1}\leq d_1,$
        \item for each $j\notin I$ we have $a_j=a_i+\delta_i+\cdots+\delta_{j-1}$ where $i\in I$ is the largest index less than $j$, and
        \item $L_C\leq b_j,\hspace{.5cm} b_k+r\leq b_{k+1}, \text{ for each } k=j,\dots,j'-1,\text{ and}\hspace{.5cm} b_{j'}\leq d_2+sa_j-R_{C}$ for each column $C = \{p_j,\dots,p_{j'}\}$.
    \end{enumerate}
\end{lemma}

\begin{proof}
    The first two conditions describing the $a$-coordinates precisely say that $(\bfa,\bfb)$ has type $\delta$ and $0\leq a_1\leq \cdots \leq a_n\leq d_1.$ The constraint that for each column $C=\{p_j,\dots,p_{j'}\}$ we have $b_k+r\leq b_{k+1}$ for each $ k=j,\dots,j'-1$ are the remaining defining conditions for $(\bfa,\bfb)\in P^{\circ}_{n,r}.$ Finally, the ``top" and ``bottom" conditions defining $P^{\square{}}_{n,r}$ are redundant except at the top and bottom points of each column, which in this case reduce to the remaining inequalities $L_C\leq b_j$ and $b_{j'}\leq d_2+sa_j-R_{C}$ respectively.
\end{proof}

Let $P^\delta_{n,r}\subseteq \R^{2n}$ be the polytope defined by the equations and inequalities in Lemma \ref{typedeltainequalities}. We study the combinatorics of $P^\delta_{n,r}$ in the case $d_1,d_2>r(n-1)$. By the second condition in Lemma \ref{typedeltainequalities}, we can write all the $a$-coordinates of points $(\bfa,\bfb)\in P^\delta_{n,r}$ in terms of those with indices in $I$, $a_{i_1},\dots,a_{i_\ell}.$ The first condition says that these points are contained in $[0,d_1]$ and between each pair of points there is a minimum increase determined by $\delta$. When $d_1$ is greater than the sum of all these minimum gaps, $\delta_1+\cdots+\delta_{n-1}$, the set of all such collections of $a$-coordinates is nonempty and forms a simplex of dimension $I$. This is always the case when $d_1>r(n-1)$ since each entry of $\delta$ is at most $r$. \\

For any collection of $a$-coordinates satisfying conditions $1$ and $2$ in the lemma and any column $C=\{p_j,\dots,p_{j'}\}$, the coordinates $b_j,\dots,b_{j'}$ are contained in $[L_C,d_2+sa_j-R_{C}]$ and between each pair of points there is an increase of at least $r$. When the length of this interval, $d_2+sa_j-R_{C}-L_C$, is greater than the sum of all these minimum gaps, $r(j'-j)$, the set of all such coordinates $b_j,\dots,b_{j'}$ is nonempty and forms a simplex of dimension $j'-j+1$, the number of points in the column. This is always the case when $d_2>r(n-1)$ since it follows from the definitions of $L_C$ and $R_C$ that $L_C\leq r(j-1)$ and $R_C\leq r(n-j')$, and so
\[ d_2+sa_j-R_{C}-L_C- r(j'-j) \geq d_2- r(n-1)+sa_j  \geq 0. \]
The size of the interval containing $b_j,\dots,b_{j'}$ varies depending on $a_j=\cdots=a_{j'}$ or equivalently on $a_i$ where $i$ is the index of the first point in the block containing this column. This analysis shows that when $d_1,d_2>r(n-1)$, the polytope $P^\delta_{n,r}$ is combinatorially equivalent to a product of simplices. In particular, we can describe its vertices and their tangent cones.

\begin{lemma}\label{verts}
    If $d_1,d_2>r(n-1)$, $P^\delta_{n,r}$ is a lattice polytope with $(|I|+1)\prod_C(|C|+1)$ vertices, where the product is over all columns $C$ for an $n$-tuple of type $\delta$. For each vertex, exactly one of the inequalities in condition $1$ of Lemma \ref{typedeltainequalities} is strict, and for each column $C$ exactly one of the inequalities in condition $3$ is strict.
\end{lemma}

We can index the vertices of $P^\delta_{n,r}$ as follows: Choose a number of blocks $k = 0,1,\dots,|I|$ and move the points in the first $k$ blocks as far left as possible, and the remaining points as far right as possible. For left points this means $a_j=\delta_1+\cdots+\delta_{j-1}$ and for right points it means $a_j=d_1-\delta_{j}-\cdots-\delta_{n-1}.$ Then, for each column $C=\{p_j,\dots,p_{j'}\}$ choose a number of points $k_C=0,1 \dots,|C|$ and move the first $k_C$ points in the column as far down as possible and the remaining points in the column as far up as possible. For bottom points this means $b_i = L_C+r(i-j)$ and for top points this means $b_i=d_2+sa_j-R_C-r(j'-i).$\\

Using the above terminology, we partition each vertex of $P^\delta_{n,r}$ considered as an $n$-tuple $(\bfa,\bfb)$ into four separate $n$-tuples of points: $(\bfa^{(1)},\bfb^{(1)})$ the points in a left block and bottom of their column, $(\bfa^{(2)},\bfb^{(2)})$ the points in a right block and bottom of their column, $(\bfa^{(3)},\bfb^{(3)})$ the points in a right block and top of their column, and $(\bfa^{(4)},\bfb^{(4)})$ the points in a left block and top of their column. \\

The equations defining a vertex cone of a polytope are obtained by removing all the inequalities that are strict at a given vertex. Fixing a vertex $v$ of $P^\delta_{n,r}$, we partition any integer point $(\bfa,\bfb)$ in the corresponding vertex cone $\mathcal K_v P^\delta_{n,r}$ into four $n$-tuples the same way as the vertex $n$-tuple. In other words, $(\bfa^{(i)},\bfb^{(i)})$ consists of the points in the $n$-tuple $(\bfa,\bfb)$ with index in $J^{(i)}.$\\

\begin{proof}[Proof of Theorem \ref{Thm:Eulerchar}]
    By the localization formula (\ref{Hirzidentity}), we have
    \[ \chi^T(X^{[n]},L_n\otimes E^r) = \sum_{n_1+\cdots+n_4=n} \left(\prod_{i=1}^4 \sigma^{(i)}_{n_i,r}(t,q)\right), \]
    and by Corollary \ref{combsigmas} the rational functions $\sigma^{(i)}_{n_i,r}(t,q)$ are generating functions summing over the integer points in the polyhedron $P^{(i)}_{n_i,r}$. In other words, $\chi^T(X^{[n]},L_n\otimes E^r)$ is equal to the sum of the generating functions of integer points in the products $P^{(1)}_{n_1,r}\times \cdots \times P^{(4)}_{n_4,r}$ for all $n_1+\cdots+n_4=n$.\\

    On the other hand, $\sigma^{\square{}}_{n,r}(t,q)$ is defined as a generating function summing over the integer points in $P^{\square{}}_{n,r}.$ These integer points are divided among the polytopes $P^\delta_{n,r}$, so we have
    \[ \sigma^{\square{}}_{n,r}(t,q) = \sum_{\delta\in \{0,1,\dots,r\}^{n-1}} \sigma^\delta_{n,r}(t,q), \]
    where 
    \[ \sigma^\delta_{n,r}(t,q) = \sum_{(\bfa,\bfb)\in P^\delta_{n,r}\cap \Z^{2n}} t^{a_1+\cdots+a_n}q^{b_1+\cdots+b_n}. \]
    By Brion's formula \cite{Brion1988} the generating function of integer points in $P^\delta_{n,r}$ is equal to the sum of those in its vertex cones. For a vertex $v\in P^\delta_{n,r}$, the vertex cone $K_v P^\delta_{n,r}$ is the polyhedron defined by all the conditions in Lemma \ref{typedeltainequalities} that are active at the vertex $v$. Brion's formula gives the identity
    \[ \sigma^\delta_{n,r}(t,q) = \sum_{\substack{v\in P^\delta_{n,r}\\\text{a vertex}}} \sum_{(\bfa,\bfb)\in K_v P^\delta_{n,r}\cap \Z^{2n}} t^{a_1+\cdots+a_n}q^{b_1+\cdots+b_n}, \]
    and so $\sigma^{\square{}}_{n,r}(t,q)$ is the sum of these over all $\delta$. To conclude the proof, we show that choices of $\delta\in \{0,1,\dots,r\}^{n-1}$, vertex $v\in P^\delta_{n,r}$, and integer point of $\mathcal K_v P^\delta_{n,r}$ are in bijection with choices of $n_1+\cdots+n_4=n$ and integer point of $P^{(1)}_{n_1,r}\times \cdots \times P^{(4)}_{n_4,r}$ in such a way that preserves the weights of the terms in the respective generating functions $\sigma^{\square{}}_{n,r}(t,q)$ and $\chi^T(X^{[n]},L_n\otimes E^r)$.\\

    Fix $\delta$ and a vertex $v$ of $P^\delta_{n,r}$, and let $(\bfa,\bfb)$ be an integer point in $\mathcal K_v P^\delta_{n,r}.$ We use the description of the vertices of $P^\delta_{n,r}$ given in Lemma \ref{verts}. Subdivide the $n$-tuple $(\bfa,\bfb)$ into left and right points depending on whether the corresponding point in the vertex $n$-tuple is in a left or right block. Further subdivide each column of $(\bfa,\bfb)$ depending on whether the corresponding point in the vertex $n$-tuple is a top or bottom point. Label in increasing lexicographic order four separate tuples of points $(\bfa^{(1)},\bfb^{(1)}),\dots,(\bfa^{(4)},\bfb^{(4)})$ which are the bottom-left, bottom-right, top-right, and top-left points in $(\bfa,\bfb)$ respectively. Let $n_1,\dots,n_4$ be the number of points in each of these tuples.\\

    Consider a column of points $\{p_j,\dots,p_{j'}\}$ in some $n$-tuple $(\bfa,\bfb)\in \mathcal K_v P^\delta_{n,r},$ and suppose that the column lies in a left block. Let $j''$ be the largest index corresponding to a left point. The equations defining $P^\delta_{n,r}$ imply that heights of the bottom points in the column are at least $\sum_{i=1}^{j-1}\max\{r-(a_j-a_i),0\}$, and similarly the heights of the top points are at most $\sum_{k=j'+1}^{j''}\max\{r-(a_k-a_j),0\}.$ Comparing these conditions to those defining $P^{(1)}_{n_i,r}$ and $P^{(4)}_{n_4,r}$, the only difference is that the sums for the lower and upper bounds are restricted to the lower and upper points respectively. For each $i<i'$ where $(a_i,b_i)$ is a top-left point and $(a_{i'},b_{i'})$ is a bottom-left point, we therefore shift $b_i$ up by $\max\{r-(a_{i'}-a_i),0\}$ and shift $b_{i'}$ down by $\max\{r-(a_{i'}-a_i),0\}$. Similarly, for each $i<i'$ where $(a_i,b_i)$ is a top-right point and $(a_{i'},b_{i'})$ is a bottom-right point, we shift $b_i$ up by $\max\{r-(a_{i'}-a_i),0\}$ and shift $b_{i'}$ down by $\max\{r-(a_{i'}-a_i),0\}$. Call the new collection of points after all the translations $(\bfa^{(i)},\tilde \bfb^{(i)})_{i=1,\dots,4}.$\\

    As discussed in the previous paragraph, the conditions defining the vertical heights of points in the transformed tuples $(\bfa^{(i)},\tilde \bfb^{(i)})_{i=1,\dots,4}$ for any fixed $a$-coordinates are exactly the same as those defining $P^{(1)}_{n_1,r}\times \cdots \times P^{(4)}_{n_4,r}$. As $\delta\in \{0,1,\dots,r\}^{n-1}$ and $v$ vary, these transformed collections of points are in bijection with the integer points each products $P^{(1)}_{n_1,r}\times \cdots \times P^{(4)}_{n_4,r}$. Furthermore, the transformation is defined by shifting certain pairs of points up and down by the same amount so the sum of the $a$ and $b$-coordinates of all the points is preserved. This shows that $\chi^T(X^{[n]},L_n\otimes E^r)$, a sum of generating functions of the products $P^{(1)}_{n_1,r}\times \cdots \times P^{(4)}_{n_4,r}$, coincides with $\sigma^{\square{}}_{n,r}(t,q)$, a sum of generating functions of cones $\mathcal K_v P^\delta_{n,r}$, completing the proof.
\end{proof}

Next we extract an explicit formula for $\chi(X^{[n]},L_n\otimes E^r)$ in the case $X=\P^1\times\P^1$. We have $\ell = \ell(\delta)$ denoting the number of blocks in an $n$-tuple of type $\delta$, as well as $L_C(\delta)$ and $R_C(\delta)$ for each column $C=\{p_j,\dots,p_{j'}\}$. We also define the statistic $|\delta|=\delta_1+\cdots+\delta_{n-1}$.

\begin{corollary}
    For $X=\P^1\times\P^1$, any line bundle $L=\Oo(d_1,d_2)$, and $r>0$, 
    \[ \chi(X^{[n]},L_n\otimes E^r) = \sum_{\delta\in \{0,1,\dots,r\}^{n-1}} { d_1-|\delta|+\ell(\delta) \choose \ell(\delta)} \prod_{C} {d_2-R_C-L_C-r|C|+r+|C|\choose |C|}, \]
    where the product is over all columns of points $C$ in an $n$-tuple of type $\delta$.
\end{corollary}

The version of this formula given in Theorem \ref{P1P1} from the introduction enumerates the columns $k=1,\dots,c(\delta)$. In the notation of the introduction, $n_k(\delta) = |C|$ and $w_k(\delta)=R_C-L_C-r|C|$ where $C$ is the $k$th column.

\begin{proof}
    Suppose first that $d_1,d_2>r(n-1)$ so that by Theorem \ref{Thm:Eulerchar} we have
    \[ \chi(X^{[n]},L_n\otimes E^r) = \#(P^{\square{}}_{n,r} \cap \Z^{2n}) =\sum_{\delta\in \{0,1,\dots,r\}^{n-1}} \#(P^\delta_{n,r} \cap \Z^{2n}). \]
    Plugging $s=0$ into the inequalities defining $P^\delta_{n,r}$ in Lemma \ref{typedeltainequalities}, we see that each $P^\delta_{n,r}$ is a product of simplices: one simplex for the horizontal positions of the blocks, and one simplex for each column of points controlling their heights.\\

    The horizontal positions of the blocks are nonnegative integers $a_{i_1},\dots,a_{i_\ell}$ such that $a_{i_{k+1}}\geq a_{i_k}+\delta_{i_k}+\cdots +\delta_{i_{k+1}-1}$ and $a_{i_\ell}+\delta_{i_\ell}+\cdots+\delta_{n-1}\leq d_1$. There are $d_1-|\delta|+\ell(\delta) \choose \ell(\delta)$ such choices of block positions. For each choice of block positions and column $C=\{p_j,\dots,p_{j'}\}$, the heights $b_j,\dots,b_{j'}$ are integers in the interval $[L_C(\delta),d_2-R_C(\delta)]$ increasing by at least $r$ in each step. There are $d_2-R_C-L_C-r|C|+r+|C|\choose |C|$ choices of such integers.\\

    This shows that for fixed $n,r>0$, the formula holds for all sufficiently large $d_1,d_2$. But both sides are polynomials in $d_1,d_2$ so the formula holds in general, completing the proof.
\end{proof}

\section{Global Sections}\label{sec:globalsections}

In this final section, we explain the sense in which the $n$-tuples we have studied correspond to global sections of line bundles on the Hilbert scheme. Let $A^r\subseteq \C[\mathbf{x,y}]$ be the spaces defined in Section \ref{sec:c2} which can be identified $A^r\simeq H^0((\C^2)^{[n]},E^r).$ As before, we equip $\C[\mathbf{x,y}]$ with the lexicographic term order with $x_1>\cdots>x_n>y_1>\cdots>y_n.$\\

Consider a smooth, projective, toric surface $X$ equipped with line bundle $L$. The global sections $H^0(X,L)$ can be identified with Laurent polynomials in $x$ and $y$ whose support is contained in $P_L$. The following result generalizes this correspondence to Hilbert schemes.

\begin{theorem}[\cite{Cavey} Proposition 4.2]
    For any smooth, projective, toric surface $X$ with line bundle $L$ and integer $r\geq0$, the global sections $H^0(X^{[n]},L_n\otimes E^r)$ can be identified with the set of polynomials in $A^r\subseteq \C[\mathbf{x,y}]$ whose support with respect to each pair of variables $(x_i,y_i)$ is contained in $P_L$.
\end{theorem}

Here we assume without loss of generality (choosing an appropriate $T$-action on $L$) that the corresponding polygon $P_L\subseteq \R^2$ is contained in the first quadrant. In the $\C^2$ case, corresponding to the entire collections of polynomials $A^r$, we were able to determine the exact sets of trailing term exponents (see Theorem \ref{C2}). For the support restricted collections of polynomials appearing in the projective case, only an upper bound for the sets of trailing terms was obtained in \cite{Cavey}. In the Hirzebruch surface case studied in the previous section, Proposition 4.7 in \cite{Cavey} states that the exponent $(\bfa,\bfb)$ appearing on the trailing term $x_1^{a_1}\cdots x_n^{a_n}y_1^{b_1}\cdots y_n^{b_n}$ of any polynomial corresponding to a section of $L_n\otimes E^r$ must be an $r$-separated $n$-tuple in $P_L$. We conjectured that every such $n$-tuple appears one of these trailing term exponents, and thanks to Theorem \ref{Thm:Eulerchar} we can prove this in the case $L_n\otimes E^r$ is ample.

\begin{corollary}
    Let $X$ be a Hirzebruch surface and $L=\Oo(d_1D_1+d_2D_2)$ with polygon $P_L$ and $r> 0$. If $L_n\otimes E^r$ is an ample line bundle on $X^{[n]}$, then the set of exponents of trailing terms of polynomials $f\in A^r$ corresponding to sections $H^0(X^{[n]},L_n\otimes E^r)$ is precisely the set of $r$-lexicographically increasing $n$-tuples of points in $P_L$.
\end{corollary}

\begin{proof}
    By \cite{Cavey} Proposition 4.7, the trailing term exponent $(\bfa,\bfb)$ of any such polynomial is an $r$-lexicographically increasing $n$-tuple in $P_L$. The number of trailing term exponents attained by polynomials corresponding to sections of $L_n\otimes E^r$ is equal to the dimension of $H^0(X^{[n]},L_n\otimes E^r)$, so it suffices to show that the number of such $n$-tuples coincides with $\dim H^0(X^{[n]},L_n\otimes E^r).$ By Theorem \ref{Thm:Eulerchar}, $\chi(X^{[n]},L_n\otimes E^r)$ is equal to the number of such $n$-tuples, and by the Frobenius splitting of $X^{[n]}$ \cite{KT}, we have $\chi(X^{[n]},L_n\otimes E^r)= \dim H^0(X^{[n]},L_n\otimes E^r)$ completing the proof.
\end{proof}

It would be interesting to give a more direct proof of the previous corollary by constructing the polynomials with each given trailing term. Such an approach would likely allow for more general results. For example, we expect that similar results should hold for $X=\P^2$ and/or non-ample line bundles $L_n\otimes E^r$ on $X^{[n]}$, but our methods relying on the specific combinatorics of the trapezoid $P_L$ and corresponding sets $P^{\square{}}_{n,r}$ and passing through the Euler characteristic do not directly extend to these cases.

\addcontentsline{toc}{chapter}{Bibliography}

\begingroup
\setstretch{1}
\bibliographystyle{plain}
\bibliography{refs.bib} 
\endgroup

\end{document}